\documentclass{amsart}
\usepackage{amsmath,amsthm,amsfonts,amssymb,amscd}
\usepackage{psfrag,graphicx}


\theoremstyle{plain}
\newtheorem{main}{Theorem}

\newtheorem{maincor}[main]{Corollary}
\newtheorem{Thm}{Theorem}[section]
\newtheorem{Lem}[Thm]{Lemma}
\newtheorem{Prop}[Thm]{Proposition}
\newtheorem{Cor}[Thm]{Corollary}
\theoremstyle{remark}

\newtheorem{Rem}[Thm] {Remark}

\long\def\begcom#1\endcom{}

\newcommand{\quand}{\quad\text{and}\quad}
\newcommand{\inter}{\operatorname{int}}

\newcommand{\diam}{\operatorname{diam}}
\newcommand{\per}{\operatorname{per}}
\newcommand{\Per}{\operatorname{Per}}

\newcommand{\Diff}{\operatorname{Diff}}
\newcommand{\Homeo}{\operatorname{Homeo}}
\newcommand{\HT}{\operatorname{HT}}
\newcommand{\HTclos}{\overline{\HT}}

\def\Orb{\operatorname{Orb}}
\def\Diff{\operatorname{Diff}}
\def\Per{\operatorname{Per}}
\def\supp{\operatorname{supp}}

\def\spec{\operatorname{sp}}

\def\cU{{\mathcal U}}

\def\cR{{\mathcal R}}

\def\cF{{\mathcal F}}

\def\cR{{\mathcal R}}

\def\cW{{\mathcal W}}
\def\vep{\varepsilon}

\def\RR{{\mathbb R}}

\def\ZZ{{\mathbb Z}}

\begin{document}

\title[Entropy conjecture away from tangencies]
      {The entropy conjecture for \\ diffeomorphisms away from tangencies}

\author{Gang Liao$^1$, Marcelo Viana$^2$, Jiagang Yang$^3$}

\thanks{$^1$
School of Mathematical Sciences, Peking University,
Beijing 100871, China}
\email{liaogang@math.pku.edu.cn}

\thanks{$^2$
IMPA, Estrada D. Castorina 110, Jardim Bot\^anico,
22460-320 Rio de Janeiro, Brazil. }

\email{viana@impa.br}

\thanks{$^3$
Departamento de Geometria, Instituto de Matem\'atica, Universidade Federal Fluminense,
Niter\'oi, Brazil}

\email{yangjg@impa.br}

\date{September, 2010}

\thanks{GL is supported by CSC-China. MV and JY are partially supported by CNPq,
FAPERJ, and PRONEX Dynamical Systems.}

\begin{abstract}
We prove that every $C^1$ diffeomorphism away from homoclinic tangencies
is entropy expansive, with locally uniform expansivity constant.
Consequently, such diffeomorphisms satisfy Shub's entropy conjecture:
the entropy is bounded from below by the spectral radius in homology.
Moreover, they admit principal symbolic extensions, and the topological
entropy and metrical entropy vary continuously with the map.
In contrast, generic diffeomorphisms with persistent tangencies are not entropy
expansive and have no symbolic extensions.
\end{abstract}

\maketitle


\section{Introduction}\label{s.introduction}

In this paper we prove that the dynamics of any diffeomorphism away from
homoclinic tangencies admits a very precise description at the topological level.
Let us begin by introducing the set-up of our results.

For each $r\ge 1$, let $\Diff^r(M)$ denote the space of $C^r$ diffeomorphisms on some compact
Riemannian manifold $M$, endowed with the $C^r$ topology.
A periodic point $p$ of $f\in\Diff^r(M)$ is \emph{hyperbolic} if the derivative
$Df^\kappa(p)$, $\kappa=\per(p)$ has no eigenvalues with norm $1$.
Then there exist $C^r$ curves $W^s(p)$  and  $W^u(q)$ - the \emph{stable} and \emph{unstable}
manifolds of $p$ - that intersect transversely at $p$ and satisfy
$$
f^{n\kappa}(q) \to p \ \text{for all } q\in W^s(p)
\quad\text{and}\quad
f^{-n\kappa}(q) \to p \ \text{for all } q\in W^u(p).
$$
A point $q\in W^s(p)\cap W^u(p)$ distinct from $p$ is a \emph{homoclinic point} associated to $p$.
The homoclinic point $q$ is \emph{transverse} if
$$
T_q M = T_q W^u(p) + T_q W^s(p).
$$
We say that $f$ has a \emph{homoclinic tangency} if there exists a non-transverse homoclinic
point associated to some hyperbolic periodic point.
The set of $C^r$ diffeomorphisms that have some homoclinic tangency will be denoted $\HT^r$.

For notational simplicity, we also write $\Diff(M)=\Diff^1(M)$ and $\HT=\HT^1$.
Our main results, that we are going to state in a while,
hold for diffeomorphisms in $\Diff(M)\setminus\HTclos$, that we call diffeomorphisms
\emph{away from tangencies}.

\subsection{Entropy conjecture}

Let $m=\dim M$ and $f_{*,k}: H_k(M,\RR)\to H_k(M,\RR)$, $0 \le k \le m$ be the action induced by
$f$ on the real homology groups of $M$. Let
$$
\spec(f_{*}) =\max_{0\leq k\leq m} \spec( f_{*,k}),
$$
where $\spec(f_{*,k})$ denotes the spectral radius of $f_{*,k}$. Shub~\cite{Shu74} has conjectured
(see also Shub, Sullivan~\cite{SS75}) that the logarithm of $\spec(f_*)$ is a lower bound for the
topological entropy of $f$:
\begin{equation}\label{eq.entttropyconjecture}
\log \spec(f_{*})\,\leq\, h(f)\quad\text{for every $f\in\Diff(M)$.}
\end{equation}
We prove that the conjecture does hold for diffeomorphisms away from tangencies:

\begin{main}\label{t.entropy conjecture}
The entropy conjecture \eqref{eq.entttropyconjecture} holds for every $f\in \Diff(M)\setminus\HTclos$.
\end{main}

This is the best result to date on the entropy conjecture in finite differentiability.
We will also comment on the behavior of diffeomorphisms with tangencies.
Before getting to that, let us briefly recall the history of this problem.

The entropy conjecture is known to hold for an open and dense subset of the space $\Homeo(M)$
of homeomorphisms, when $\dim M \neq 4$.
In fact, by Palis, Pugh, Shub, Sullivan~\cite{PPSS75}, the conjecture always holds for an open and
dense subset of any stable connected component of $\Homeo(M)$.
When the dimension is different from $4$ all connected components are stable,
by Kirby, Siebenmann~\cite{KS69}, and that is how one gets the previous statement.

Manning~\cite{Mann75} proved that the weaker inequality $\log\spec(f_{*,1})\le h(f)$
always holds for homeomorphisms in any dimension. Using Poincar\'e duality, one
deduces the full statement of the entropy conjecture for homeomorphisms on manifolds
with $\dim M \le 3$. The conjecture is also known to hold for homeomorphisms on
any infra-nilmanifold, by Marzantowicz, Misiurewicz, Przytycki~\cite{MP77b,MP08}.

Weaker versions of the conjecture, where one replaces the spectral radius of $f_*$ by
other topological invariants, have been proved in great generality.
Bowen~\cite{Bow78} showed that $\log \gamma_1 \le h(f)$ for every homeomorphism,
where $\gamma_1$ is the growth rate of the fundamental group. This is a strengthening
of Manning's result mentioned previously. Ivanov~\cite{Iva82} proved that
the asymptotic Nielsen number is also a lower bound for the topological entropy,
for every homeomorphism. Moreover, Misiurewicz, Przytycki~\cite{MP77a} showed
that the topological entropy of every homeomorphism is bounded from below by the
logarithm of the degree. For local diffeomorphisms a proof can be given using the
Perron-Fr\"obenius operator (see Oliveira, Viana~\cite{OV08}).

On the other hand, Shub~\cite{Shu74} exhibited a Lipschitz (piecewise affine)
counterexample to the entropy conjecture: while the spectral radius is strictly
positive, the topological entropy vanishes. Thus, some smoothness is necessary
for a general (not just generic) statement. A major progress was the proof,
by Yomdin~\cite{Yom87}, that the entropy conjecture is true for every $C^\infty$ diffeomorphism.
The main ingredient is a relation between topological entropy $h(f)$ and the growth rate
$v(f)$ of volume under iteration by a diffeomorphism. For $C^\infty$ diffeomorphisms the two
numbers actually coincide (that is false in finite differentiability). The entropy conjecture
is a consequence, because $\log\spec(f_{*}) \le v(f)$ for any $C^1$ diffeomorphism $f$.

The entropy conjecture has also been established for certain classes of systems with
hyperbolicity properties: Anosov diffeomorphisms and, more generally,
Axiom A diffeomorphisms with no cycles (Shub, Williams~\cite{SW75},
Ruelle, Sullivan~\cite{RS75}), and partially hyperbolic systems with
one-dimensional center bundle (Saghin, Xia~\cite{SX10}). All of these systems are
away from tangencies, of course.

\subsection{Entropy expansiveness and continuity of entropy}

Theorem~\ref{t.entropy conjecture} will be deduced from the following result:

\begin{main}\label{t.entropy expansiveness}
Every diffeomorphism $f\in \Diff(M)\setminus \HTclos$ is entropy expansive.
\end{main}

\begin{Rem}\label{r.converse}
In contrast, there is a residual subset $\cR$ of $\Diff(M)$ such that any $f\in \cR\cap \HTclos$
is \emph{not} entropy expansive. This is related to results of Downarwicz, Newhouse~\cite{DN05}.
 A proof will appear in Section~\ref{ss.proofs}.
\end{Rem}

The notion of entropy expansiveness will be recalled in Section~\ref{s.entropytheory}.
It was first introduced by Bowen~\cite{Bow72b}, who observed that for entropy expansive
maps the metric entropy function (defined in the space of invariant probabilities)
$$
\mu \mapsto h_\mu(f)
$$
is upper semi-continuous. In particular, for such maps there always exists some
measure of maximum entropy. In view of these observations,
Theorem~\ref{t.entropy expansiveness} has the following direct consequence:

\begin{maincor}\label{c.main}
For any $f\in \Diff(M)\setminus \HTclos$ the entropy function $\mu\to h_{\mu}(f)$
is upper semi-continuous and, thus, there is some invariant probability $\mu$
with $h_\mu(f)=h(f)$.
\end{maincor}

The first examples of $C^r$ diffeomorphisms without measures of maximum entropy
were given by Misiurewicz~\cite{Mis73}, for each $1 \leq r < \infty$.
He also introduced a weaker condition, called asymptotic entropy expansiveness,
that suffices for upper semi-continuity of the metric entropy function.
In addition, Misiurewicz~\cite{Mis73} gave examples of $C^r$ diffeomorphisms,
$1 \leq r <\infty$ where the topological entropy function
$$
f \mapsto h(f).
$$
fails to be upper semi-continuous.
For $C^\infty$ diffeomorphisms, Newhouse~\cite{New89} proved that the metric entropy
function is always upper semi-continuous, and Yomdin~\cite{Yom87} proved upper
semi-continuity of the topological entropy function. Newhouse's result has been improved
by Buzzi~\cite{Buz97}, who showed that every $C^\infty$ diffeomorphism is asymptotically
entropy expansive. Yomdin's semi-continuity result also extends to every $C^1$ diffeomorphism
away from tangencies:

\begin{main}\label{t.upper-semi-continuity of topological entropy}
The topological entropy is upper semi-continuous on $\Diff(M)\setminus\HTclos$.
\end{main}

Closing this section, let us observe that the metric entropy function is usually not
lower semi-continuous.
Indeed, by the ergodic closing lemma of Ma\~n\'e~\cite{Man82}, there is a residual
subset $\cR_1$ of $\Diff(M)$ such that for every $f\in\cR_1$ every ergodic invariant
measure is approximated by invariant measures supported on periodic orbits.
Thus, for every $f\in \cR_1$, either $h(f)=0$ or the metric entropy function fails to
be lower semi-continuous.
For maps on compact surfaces without boundary, it follows from Katok~\cite{Ka80}
that the topological entropy function is lower semi-continuous on $\Diff^r(M)$,
for all $r>1$. By Gromov~\cite{Gr85}, this does not extend to surfaces with boundary.

\subsection{Symbolic extensions}

A \emph{symbolic extension} of a map $f:M\to M$ is a subshift $\sigma:Y \to Y$
over a finite alphabet, together with a continuous surjective map $\pi:Y \to M$
such that $f\circ\pi=\pi\circ \sigma$. Markov partitions for uniformly hyperbolic
systems (Bowen~\cite{Bow75a}) are the classical prototype. In general, a
symbolic extension may carry a lot more dynamics than the original map $f$.
We call a symbolic extension \emph{principal} if it is minimal in this regard:
$h_\mu(f) = h_{ext}^{\pi}(\mu)$, where $h_{ext}^{\pi}(\mu)$ is the supremum
of the entropy $h_{\nu}(\sigma)$ of the shift $\sigma$ over all invariant probabilities
$\nu$ such that $\pi_{*}\nu=\mu$.

\begin{maincor}\label{c.principal extension}
Any $f\in \Diff(M)\setminus \HTclos$ admits a principal symbolic extension.
\end{maincor}

This follows directly from Theorem~\ref{t.entropy expansiveness} together with
the observation by Boyle, Fiebig, Fiebig~\cite{BFF02} that every asymptotically
entropy expansive diffeomorphism admits a principal symbolic extension.

Let us also point out that D\'\i az, Fisher, Pacifico, Vieitez~\cite{DFPV,PV08} have, recently,
constructed principal symbolic extensions for partially hyperbolic diffeomorphism admitting
an invariant splitting into one dimensional subbundles. Indeed, they prove that such maps
are entropy expansive. This is in contrast with previous work of
Downarwicz, Newhouse~\cite{DN05}, based on the theory developed by
Boyle, Downarwicz~\cite{BD04}, where it is shown that nonexistence of symbolic
extensions is typical on the closure of the set of  area preserving diffeomorphisms
with homoclinic tangencies. Also very recently, Catalan, Tahzibi~\cite{CTh} proved
non-existence of symbolic extensions for generic symplectic diffeomorphisms
outside the Anosov domain. In this setting, they also find lower bounds for the
topological entropy in terms of the eigenvalues at periodic points.

\section{Entropy theory}\label{s.entropytheory}

Here we recall some basic facts about entropy. See Bowen~\cite{Bow72b} and
Walters~\cite{Wa82} for more information. Moreover, we propose an alternative definition
of entropy expansiveness, in terms of invariant measures (almost entropy expansiveness),
that will be useful in the sequel.

\subsection{Definitions and statements}

Throughout, $f:M\to M$ is a continuous map on a compact metric space $M$.
Let $K$ be a subset of $M$. For each $\vep>0$ and $n\ge 1$, we consider the following objects.
The \emph{dynamical ball} of radius $\vep>0$ and length $n$ around $x\in M$ is the set
$$
B_n(x,\vep)=\{y\in M: d(f^j(x),f^j(y))\le\vep \text{ for every } 0 \le j < n\}.
$$
A set $E\subset M$ is \emph{$(n,\vep)$-spanning for $K$} if for any $x\in K$ there is $y\in E$
such that $d(f^ix,f^iy)\leq\vep$ for all $0\leq i<n$. In other words, the dynamical balls
$B_n(y,\vep)$, $y\in E$ cover $K$. Let $r_n(K,\vep)$ denote the smallest cardinality of any
$(n,\vep)$-spanning set, and
$$
r(K,\vep)=\limsup_{n\to+\infty}\frac{1}{n}\log r_n(K,\vep).
$$
A set $F\subset K$ is \emph{$(n,\vep)$-separated} if for any distinct points $x$ and $y$ in $F$
there is $0 \le i < n$ such that $d(f^ix,f^iy)>\vep$.  That is, no element of $F$ belongs to the
dynamical ball $B_n(y,\vep)$ of another.
Let $s_n(K,\vep)$ denote the largest cardinality of any $(n,\vep)$-separated set, and
$$
s(K,\vep)=\limsup_{n\to+\infty}\frac{1}{n}\log s_n(K,\vep).
$$
The \emph{topological entropy of $f$ on $K$} is defined by
$$
h(f,K)
= \lim_{\vep \to 0} s(K,\vep)
= \lim_{\vep \to 0} r(K,\vep).
$$
The \emph{topological entropy of $f$} is defined by $h(f)=h(f,M)$.
Given any finite open cover $\beta$ of $M$, let
\begin{equation}\label{eq.hbetaf}
h(f,\beta)=\lim_{n\to\infty} \frac 1n \log |\beta^n|
          =\inf_{n\ge 1} \frac 1n \log |\beta^n|,
\end{equation}
where $\beta^n=\{A_0 \cap f^{-1}A_1\cap\cdots\cap f^{-n+1}A_{n-1)}: A_i \in \beta \text{ for } 0\leq i\leq n-1\}$
and $|\beta^n|$ is the smallest cardinality of a subcover of $\beta^n$.
The topological entropy $h(f)$ coincides with the supremum of $h(f,\beta)$ over all finite open covers.


\begin{Rem}\label{equivalence of definitions of entropy}
If $\diam(\beta)<\varepsilon$ then $r_n(M,\vep)\leq s_n(M,\vep)\le|\beta^n|$ for every $n$.
Hence, $r(M,\vep)\leq s(M,\vep)\leq h(f,\beta)$.
\end{Rem}

\begin{Lem}[Bowen~\cite{Bow72b}]\label{l.Bowen}
Let $0=t_0<t_1<\cdots<t_{r-1}<t_r=n$ and, for $0\le i < r$, let
$E_i$ be a $(t_{i+1}-t_i,\vep)$-spanning set for $f^{t_i}(F)$.
Then $$r_n(F,2\vep)\leq \prod_{0\leq i <r}\#(E_i).$$
\end{Lem}

%

Now let $\mu$ be an $f$-invariant probability measure and $\xi=\{A_1,\cdots,A_k\}$ be a finite partition
of $M$ into measurable sets. The \emph{entropy of $\xi$ with respect to $\mu$} is
$$
H_{\mu}(f,\,\xi)=-\sum_{i=1}^{k}\mu(A_i)\log \mu(A_i).
$$
The \emph{entropy of $f$ with respect to $\xi$ and $\mu$} is given by
$$
h_{\mu}(f,\,\xi)=\lim_{n\to+\infty}\frac{1}{n}\log H_{\mu}(f,\,\xi^n).
$$
Finally, the \emph{entropy of $f$ with respect to $\mu$} is given by
$$
h_{\mu}(f)=\sup_{\xi}h_{\mu}(f,\xi),
$$
where $\xi$ ranges over all finite measurable partitions of $M$.

For each $x\in M$ and $\vep>0$, let $B_{\infty}(x,\vep)=\{y: d(f^n(x),f^n(y))\le\vep\text{ for } n\ge 0\}$.
The map $f$ is \emph{entropy expansive} if there exists $\vep>0$ such that
$$
\sup_{x\in M}h(f,B_{\infty}(x,\vep))=0.
$$
Then we say that $f$ is \emph{$\vep$-entropy expansive}.
When $f$ is a homeomorphism, one may replace $B_{\infty}(x,\vep)$ by
$B^{\pm}_{\infty}(x,\vep) =\{y: d(f^n(x),f^n(y))\le\vep\text{ for } n\in\ZZ\}$:
indeed, Bowen~\cite[Corollary~2.3]{Bow72b} gives that
$\sup_x h(f,B_{\infty}(x,\vep)) = \sup_x h(f, B^{\pm}_{\infty}(x,\vep)) $ for every $\vep>0$.

\begin{Lem}\label{l.upper continuous topological entropy}
Let $\cW\subset\Homeo(M)$ and $\vep>0$ be such that every $f\in\cW$ is $\vep$-entropy expansive.
Then the topological entropy $f\mapsto h(f)$ is upper semi-continuous on $\cW$.
\end{Lem}

\begin{proof}
Bowen~\cite[Theorem~2.4]{Bow72b} asserts that $h(f)=r(M,\vep)$ if $f$ is
$\vep$-entropy expansive. Then, by Remark~\ref{equivalence of definitions of entropy},
we have $h(f)=h(f,\beta)$ for every $f\in\cW$ and every open covering $\beta$ of $M$ with
$\diam\beta<\vep$. Let $\beta$ be fixed. It is easy to see from the definition
\eqref{eq.hbetaf} that the map $f \mapsto h(f,\beta)$ is upper semi-continuous
(because it is an infimum of upper semi-continuous functions). This gives the
claim.
\end{proof}


Let $f$ be a homeomorphism and $\mu$ be any $f$-invariant probability measure.
Given $\vep>0$, we say that $f$ is \emph{$(\mu,\vep)$-entropy expansive} if
\begin{equation}\label{eq.almost}
h(f, B^{\pm}_{\infty}(x,\vep))=0 \quad\text{for $\mu$-almost every $x\in M$.}
\end{equation}
We say that $f$ is \emph{$\vep$-almost entropy expansive} if it is $(\mu,\vep)$-entropy expansive
for any invariant probability measure $\mu$. It is clear that $\vep$-entropy expansiveness implies
$\vep$-almost entropy expansiveness. The converse is important for our purposes:

\begin{Prop}\label{total implies entropy expansive}
If $f$ is $\vep$-almost entropy expansive then $f$ is $\vep$-entropy expansive.
\end{Prop}

This follows from a stronger result, Proposition~\ref{from essential to real}, that we present
in the next section. The notion of almost entropy expansiveness extends to non-invertible maps,
with $B_{\infty}(x,\vep)$ instead of $B^{\pm}_{\infty}(x,\vep)$ in the definition \eqref{eq.almost}.
Proposition~\ref{from essential to real} remains true, with the same change in the hypothesis,
and so Proposition~\ref{total implies entropy expansive} also extends to the non-invertible case.

\subsection{Entropy expansiveness from almost entropy expansiveness}

Let $f$ be a homeomorphism. We denote
$B_n^{\pm}(x,\vep)=\{z\in M: d(f^j(x),f^j(z))\le\delta\text{ for } |j|<n\}$, for each $x\in M$ and $\vep>0$.
Proposition~\ref{total implies entropy expansive} is the particular case $a=0$ of

\begin{Prop}\label{from essential to real}
Given $a\ge 0$, if $h(f,B^{\pm}_\infty(x,\vep))\leq a$ for $\mu$-almost every $x\in M$
and every $f$-invariant probability $\mu$, then $h(f,B_\infty(x,\vep))\leq a$ for every $x\in M$.
\end{Prop}

\begin{proof}
Suppose that $h(f,B_\infty(x_0,\vep)) > a$ for some $x_0\in M$.
Fix constants $a_1$ and $a_2$ such that $h(f,B_\infty(x_0,\vep)) > a_1 > a_2 > a$.
Then, there exists $\delta>0$, arbitrarily small, and a subsequence $(m_i)_i\to\infty$
such that
\begin{equation}\label{eq.rmi}
r_{m_i}(B_\infty(x_0,\vep),\delta)> e^{a_1 m_i}
\quad\text{for every $i$.}
\end{equation}
Write $\mu_{m_i}= ({1}/{m_i}) \sum_{j=0}^{m_i-1}\delta_{f^j(x_0)}$.
By compactness, $(\mu_{m_i})_i$ may be taken to converge, in the weak$^*$ topology,
to some invariant measure $\mu$. For each $n\ge 1$, denote
$$
\Gamma_n=\{x\in M: r_m(B^{\pm}_\infty(x,\vep),\delta/4)< e^{a_2 m} \text{ for any } m\ge n\}.
$$
These sets form an increasing sequence and, as long as $\delta$ is sufficiently small,
the hypothesis implies that $\cup_n \Gamma_n$ has full $\mu$-measure.
So, we may choose an increasing sequence of compact sets $\Lambda_n\subset \Gamma_n$
such that $\mu(\cup_n \Lambda_n)=1$.
For each $n\ge 1$ and $y\in \Lambda_n$, let $E_n(y)$ be an $(n,\delta/4)$-spanning set
for $B^{\pm}_\infty(y,\vep)$ with $\# E_n(y)< e^{a_2 n}$. Then
$$
U_n(y)
  = \bigcup_{z\in E_n(y)} B_n(z,\delta/2)
$$
is a neighborhood of the compact set $B^{\pm}_\infty(y,\vep)$.
So, we may choose $N=N_n(y)$ and an open neighborhood $V_n(y)$ of $y\in\Lambda_n$
such that $B^{\pm}_{N}(u,\vep) \subset U_n(y)$ for every $u\in V_n(y)$.
Choose $y_1, \dots, y_s \in\Lambda_n$ such that the $V_n(y_i)$, $i=1, \dots, s$ cover the compact
set $\Lambda_n$. Then let $W_n=\bigcup_{1\leq i \leq s} V_n(y_i)$ and
$L(n)=\max\{n, N_n(y_1), \dots, N_n(y_s)\}$.
The fact that $W_n$ is an open neighborhood of $\Lambda_n$ ensures that
\begin{equation}\label{eq.Wn}
\lim_{i\to\infty}\mu_{m_i}(W_n)\geq \mu(W_n)\geq \mu(\Lambda_n).
\end{equation}
Consider the sequence of integers $0=t_0<t_1<\cdots<t_r=m_i$ defined as follows.
Let $j\ge 0$ and suppose that $t_0,\cdots, t_j$ have been defined. Then, take
$$
t_{j+1} = \left\{\begin{array}{ll}
t_j+n & \text{if $f^{t_j}(x_0)\in W_n$ and $L(n) \le t_j < m_i-L(n)$} \\
t_j+1 & \text{otherwise.}
\end{array}\right.
$$
Write $\{t_0,t_1,\cdots, t_r\}$ as a disjoint union $A \cup B$, where $t_j \in A$ if $f^{t_j}(x_0)\in W_n$
and $L(n) \le t_j < m_i-L(n)$ and $t_j\in B$ otherwise. For $t_j \in A$, choose $s_j\in\{1, \dots, s\}$
such that $f^{t_j}(x_0) \in V_n(y_{s_j})$. Then
$$
f^{t_j}(B_{m_i}(x_0,\vep)) \subset B_{L(n)}^{\pm}(f^{t_j}(x_0), \vep) \subset U_n(y_{s_j})
$$
and so $f^{t_j}(B_{m_i}(x_0,\vep))$ is $(n,\delta/2)$-spanned by $E(y_{s_j})$.
Fix any $\delta/2$-dense subset  $E_*$ of the ambient space $M$.
Then $f^{t_j}(B_{m_i}(x_0,\vep))$ is $(1,\delta/2)$-spanned by $E_*$ for any $t_j\in B$.
So, Lemma~\ref{l.Bowen} applies to give
$$
r_{m_i}(B_{m_i}(x_0,\vep),\delta)
\leq \prod_{t_j \in A}\# E(y_{s_j}) \cdot  (\# E_*)^{\# B}
 \leq e^{a_2 n \# A} \cdot \kappa^{\#B},
$$
where $\kappa=\#E_*$. The definitions also imply that $n\#A \le m_i$ and
$$
\#B
\leq \#\{0\leq j < m_i: f^j(x_0) \notin W_n\}+2L(n)
 = (1-\mu_{m_i}(W_n)) m_i + 2L(n).
$$
Replacing in the previous inequality, we find that
$$
\begin{aligned}
r_{m_i}(B_{m_i}(x_0,\vep),\delta)
& \leq e^{a_2 m_i}  \cdot \kappa^{(1-\mu_{m_i}(W_n))m_i+2L(n)}\\
& = \exp\Big(m_i\big(a_2 + (1-\mu_{m_i}(W_n))\log\kappa+\frac{2L(n)}{m_i}\log\kappa\big)\Big)
\end{aligned}
$$
Fix $n$ large enough so that $1 - \mu(\Lambda_n)<(a_1-a_2)/(2\log\kappa)$.
Then, using \eqref{eq.Wn}, take $m_i$ to be large enough so that $1 - \mu_{m_i}(W_n)$ and
$2L(n)/m_i$ are both smaller than $(a_1-a_2)/(2\log\kappa)$.
Then the previous inequality yields
$$
r_{m_i}(B_\infty(x_0,\vep),\delta)
\leq r_{m_i}(B_{m_i}(x_0,\vep),\delta)
< e^{a_1 m_i},
$$
contradicting \eqref{eq.rmi}. This contradiction completes the proof of the proposition.
\end{proof}

\section{Almost entropy expansiveness}

Here we prove that every diffeomorphism away from tangencies is robustly almost entropy expansive:

\begin{Thm}\label{t.entropyexpansive}
Every diffeomorphism away from tangencies admits a $C^1$ neighborhood $\cU$ and some constant
$\vep>0$ such that $h(g, B^{\pm}_{\infty}(x,\vep))=0$ for every $g\in \cU$, every $g$-invariant
probability $\mu$, and $\mu$-almost every $x\in M$.
\end{Thm}

In view of Proposition~\ref{total implies entropy expansive}, this implies that every such
diffeomorphism is robustly entropy expansive, with locally uniform expansiveness constant:

\begin{Cor}\label{c.entropyexpansive}
Every diffeomorphism away from tangencies admits a $C^1$ neighborhood $\cU$ and some constant
$\vep>0$ such that every $g\in\cU$ is $\vep$-entropy expansive.
\end{Cor}

\subsection{Preparatory remarks}

Let $\Lambda\subset M$ be a compact set invariant under $f$.
Let $T_\Lambda M=E^1\oplus \cdots \oplus E^k$ be a splitting of the tangent bundle over $\Lambda$
into $Df$-invariant subbundles (some of the $E^j$ may reduce to $\{0\}$).
Given an integer $L\ge 1$, the splitting is called \emph{$L$-dominated} if for every $i<j$,
every $x\in \Lambda$, and every pair of non-zero vectors $u\in E^i_x$ and $v\in E^j_x$, one has
$$
\frac{\|Df_x^L(u)\|}{\|u\|} < \frac{1}{2}\frac{\|Df_x^L(v)\|}{\|v\|}.
$$

In the sequel we focus on the case of dominated splittings $T_\Lambda M = E^1\oplus E^2\oplus E^3$
into three subbundles. Write $E^{ij}=E^i\oplus E^j$ for $i\neq j$.
Given a foliation $\cF$ and a point $y$ in the domain, we denote by $\cF(y)$ the leaf through $y$
and by $\cF(y,\rho)$ the neighborhood of radius $\rho>0$ around $y$ inside the leaf.
Following Burns, Wilkinson~\cite{BW10} we avoid assuming dynamical coherence by using locally
invariant (``fake'') foliations, a construction that goes back to Hirsch, Pugh, Shub~\cite{HPS77}.
For any $L$-dominated splitting over any invariant set of a diffeomorphism in some small
neighborhood of $f$, the angles between the invariant subbundles are bounded from zero by a
constant that depends only on $L$. This simple observation allows us to get the
Hirsch, Pugh, Shub statement in a somewhat more global form:

\begin{Lem}\label{Fake foliations2}
For any $f\in\Diff(M)$, $L\ge 1$, and $\zeta>0$ there is a $C^1$ neighborhood $\cU_f$
of $f$ and real numbers $\rho > r_0 >0$ with the following properties.
For any $g\in\cU_f$ let $\Lambda_g$ be a $g$-invariant compact set such that the
tangent space over $\Lambda_g$ admits an  $L$-dominated splitting
$T_{\Lambda_g}M=E^1_g \oplus E^2_g \oplus E^3_g$.
Then, the neighborhood $B(x,\rho)$ of every $x\in \Lambda_g$ admits foliations
$\cF^1_{g,x}$, $\cF^2_{g,x}$, $\cF^3_{g,x}$, $\cF^{12}_{g,x}$, $\cF^{23}_{g,x}$
such that for every $y\in B(x,r_0)$ and $*\in \{1,2,3,12,23\}$:
\begin{enumerate}
\item
the leaf $\cF^*_{g,x}(y)$ is $C^1$ and $T_y\big(\cF^*_{g,x}(y)\big)$
lies in a cone of width $\zeta$ about $E_x^*$;
\item
$g(\cF_{g,x}^*(y,r_0))\subset\cF_{g,x}^*(g(y)) $
and $g^{-1}(\cF_{g,x}^*(y,r_0))\subset \cF_{g,x}^*(g^{-1}(y))$;
\item
$\cF_{g,x}^1$ and $\cF_{g,x}^2$ subfoliate $\cF_{g,x}^{12}$ and $\cF_{g,x}^2$ and
$\cF_{g,x}^3$ subfoliate $\cF_{g,x}^{23}$.
\end{enumerate}
\end{Lem}


For simplicity, let us drop the reference to $g$ in the notations for the invariant
subbundles and foliations. Lemma~\ref{Fake foliations2} allows us to define product
structures on the $r$-neighborhood of every point $x\in\Lambda_g$, as follows.
For $y$, $z\in B(x,\rho)$, write
\begin{itemize}
\item $[y,z]_{1,2}=a$ if $z \in \cF_x^{12}(y)$ and $\cF_x^1(y)$ intersects $\cF_x^2(z)$
at $a\in B(x,\rho)$;
\item $[y,z]_{12,3}=a$ if $\cF_x^{12}(y)$ intersects $\cF_x^3(z)$ at $a\in B(x,\rho)$.
\end{itemize}
Analogously, one defines $[y,z]_{2,3}$ and $[y,z]_{1,23}$.
By transversality (Lemma~\ref{Fake foliations2}(1)), in each case the intersection
point $a$ is unique when it exists. Moreover, one can find $r_1 \in (0,r_0]$,
independent of $g$, $\Lambda_g$, and $x$, such that $[y,z]_*$ is well defined whenever
$y$ and $z$ belong to $B(x, r_1)$. Moreover, for any $y\in B(x,r_1)$ there are points
$y_*\in\cF^*_x(x)$, for each $*\in\{1,3,12,23\}$, such that
\begin{equation}\label{eq.yyy}
[y_{3},y_{12}]_{12,3} = y = [y_{23},y_{1}]_{1,23}.
\end{equation}
Part (1) of Lemma~\ref{Fake foliations2} ensures (for sufficiently small $\zeta$) that the
locally invariant foliations $\cF^*_x$ are transverse, with angles uniformly bounded from
below. Thus, there exists $l>0$, independent of $g$, $\Lambda_g$, and $x$, such that
\begin{equation}\label{eq.lipschitz}
y_* \in \cF^*_x(x,lr)\quad\text{for all $*\in\{1,3,12,23\}$ and}
\end{equation}
\begin{equation}\label{eq.center leaf}
\left\{\begin{array}{l}
y_1=x \Rightarrow y=y_{23}\in\cF^{23}_x(x,lr) \\
y_3=x \Rightarrow y=y_{12}\in\cF^{12}_x(x,lr)\end{array}
\right\} \Rightarrow y \in \cF^{12}_x(x,lr) \cap \cF^{23}_x(x,lr) = \cF^2_x(x,lr)
\end{equation}
for any $y \in B^{\pm}(x,r)$ with $lr < r_1$. Moreover, $y\in B_\infty(x,r)$ implies
\begin{equation}\label{eq.Infty ball}
(f^j(y))_* \in \cF^*_{f^j(x)}(f^j(x),lr) \quand f^j(y_*)=(f^j(y))_*
\end{equation}
for all $j\in \ZZ$ and $*\in\{1,3,12,23\}$ (by local invariance of the foliations).

The next proposition improves on a main result of Yang~\cite{yang}, see also
Crovisier~\cite{Cro2}, and is the key step for Theorem~\ref{t.entropyexpansive}.
The proof is given in Section~\ref{s.proofofdomination}.

\begin{Prop}\label{p.dominatedsplitting}
Let $f:M\to M$ be a diffeomorphism away from tangencies. Then there exist
$\lambda_0>0$, $L_0\ge 1$, and a $C^1$ neighborhood $\cU_0$ of $f$, such that,
given any $g\in\cU_0$, the support of any ergodic $g$-invariant measure $\mu$
admits an $L_0$-dominated splitting $T_{\supp\mu}M=E^1\oplus E^2 \oplus E^3$
with $\dim(E^2)\leq 1$ and, for $\mu$-almost every point $x$,
\begin{equation}\label{eq.asymptotic}
\begin{aligned}
\lim_{n\to\infty} & \frac{1}{n}\sum_{i=1}^{n}\log\|Dg^{L_0} \mid E^1_{g^{-iL_0}(x)}\|\le-\lambda_0
\quand\\
\lim_{n\to\infty} & \frac{1}{n}\sum_{i=1}^{n}\log\|Dg^{-L_0} \mid E^3_{g^{iL_0}(x)}\|\le-\lambda_0.
\end{aligned}
\end{equation}
\end{Prop}

\subsection{Proof of Theorem \ref{t.entropyexpansive}}

Let $\lambda_0$, $L_0$, and $\cU_0$ be as in Proposition \ref{p.dominatedsplitting}.
Fix $\delta>0$ with $2\delta<\lambda_0$ and then let $\zeta>0$ and $r_*>0$ be sufficiently
small so that, for any $g\in \cU_0$, we have
\begin{equation}\label{eq.distortion}
 e^{-\delta}
  \le \frac{\|Dg^{L_0}(x)u\|}{\|Dg^{L_0}(y)v\|}
  \le e^{\delta}
  \quand
  e^{-\delta}
  \le \frac{\|Dg^{-L_0}(x)u\|}{\|Dg^{- L_0}(y)v\|}
  \le e^{\delta}
\end{equation}
whenever $d(x,y)\le r_*$ and $\angle(u,v)\le \zeta$ (begin by choosing some local trivialization
of the tangent bundle). Let $\cU_f$, $r_1$, and $l$ be as in Lemma~\ref{Fake foliations2}
and the comments following it. Take $\cU=\cU_0 \cap \cU_f$ and $\vep=\min\{r_1/l,r_*/l\}$.
We are going to prove that the conclusion of Theorem~\ref{t.entropyexpansive}
holds for these choices.

By ergodic decomposition, it is no restriction to suppose that the measure $\mu$ is ergodic.
Given $x\in M$, denote $x_i=g^{iL_0}(x)$ for each $i\in\ZZ$. Let $\Gamma$ be the set of points
$x\in\supp\mu$ such that
$$
\begin{aligned}
\lim_{n\to\infty}\frac{1}{n}\sum_{i=1}^{n}\log\|Dg^{L_0} \mid E^1_{x_{-i}}\| \le -\lambda_0
\quand
\lim_{n\to\infty}\frac{1}{n}\sum_{i=1}^{n}\log\|Dg^{-L_0} \mid E^3_{x_i}\| \le -\lambda_0.
\end{aligned}
$$
Proposition \ref{p.dominatedsplitting} asserts that $\mu(\Gamma)=1$. Take $x\in\Gamma$ and
$y\in B(x,\vep)$, and then let $y_*\in\cF_x^*$, $*\in\{1,3,12,23\}$ be as in \eqref{eq.yyy}.
We claim that
\begin{equation}\label{eq.y1y3}
y_1 = x = y_3 \quad\text{for every } y\in B^{\pm}_\infty(x,\vep).
\end{equation}
If $E^3=\{0\}$ the leaf $\cF^3_x(x)$ reduces to $\{x\}$ and there is nothing to prove.
So, let us assume that $E^3$ is non-trivial.

\begin{Lem}[Pliss~\cite{Pli72}]\label{Pliss}
Given $a_* \le c_2 < c_1$ there exists $\theta=(c_1-c_2)/(c_1-a_*)$ such that,
given any real numbers $a_1, \cdots, a_{N}$ with
$$
\sum_{i=1}^N a_i\leq c_2 N \quand a_i\geq a_* \text{ for every $i$,}
$$
there exist $l>N\theta$ and $1 \le n_1<\cdots<n_l\leq N$ such that
$$
\sum_{i=n+1}^{n_j}a_i\leq c_1(n_j-n) \quad\text{for all}\quad 0\leq n<n_j \quand j=1,\cdots,l.
$$
\end{Lem}

Take $a_*=\min\{\log\|Dg^{-L_0}(x)\|: g\in \cU \text{ and } x\in M\}$ and note that $a_*\le -\lambda_0$.
Let $-\lambda_0 < c_2 < c_1 = -\lambda_0+\delta$.
Applying Lemma~\ref{Pliss} to $a_i=\log\|Dg^{-L_0} \mid E^3_{x_{i}}\|$ and large values of $N$,
we find an infinite sequence  $1 \leq n_1 < n_2 < \cdots < n_j < \cdots$ such that
$$
\sum_{t=n+1}^{n_j}\log\|Dg^{-L_0} \mid E^3_{x_{i}}\|\leq (-\lambda_0+\delta)(n_j-n)
\quad\text{for every $0 \leq n < n_j$.}
$$
By Lemma \ref{Fake foliations2}, the relation \eqref{eq.distortion}, and our choice of $\vep$,
$$
e^{-\delta}
\le \frac{\|Dg^{-L_0}\mid T_{z}\cF_{x_i}^3(x_i)\|}
       {\|Dg^{-L_0}\mid T_{x}\cF_{x_i}^3(x_i)}\|
\le e^{\delta}
\quad\text{for every $z\in\cF^1_x(x_i,l\vep)$ and $i\in\ZZ$.}
$$
From these two relations one gets that
$$
g^{(n-n_j)L_0}(\cF^3_{x_{n_j}}(x_{n_j},l\vep))
\subset \cF^3_{x_{n}}(x_{n},e^{(n_j-n)(-\lambda_0+2\delta)}l\vep).
$$
for every $0 \le n < n_j$ and, in particular,
\begin{equation}\label{eq.contracts}
g^{-n_j L_0}(\cF^3_{x_{n_j}}(x_{n_j},l\vep))
\subset \cF^3_{x}(x,e^{n_j(-\lambda_0+2\delta)}l\vep).
\end{equation}
Let $y\in B^{\pm}_{\infty}(x,\vep)$. By \eqref{eq.Infty ball} and our choice of $\vep$,
the point $g^{iL_0}(y_3)=(g^{iL_0}(y))_3$ belongs to $\cF^3(x_i,l\vep)$ for every $i$.
In particular, $y_3$ belongs to the intersection of all $g^{-n_jL_0}(\cF^3_{x_{n_j}}(x_{n_j},l\vep))$
over all $j$. By \eqref{eq.contracts}, this intersection reduces to $\{x\}$.
So, $y_3=x$ as  claimed in \eqref{eq.y1y3}. The proof that $y_1=x$ is entirely analogous,
and so the proof of the claim is complete. Together with the relations \eqref{eq.center leaf}
and \eqref{eq.Infty ball}, this gives that
$$
g^j(B^{\pm}_\infty(x,\vep))\subset \cF^2_{g^j(x)}(g^j(x),r_1) \quad\text{for any $j\in \ZZ$.}
$$
Observe that the $\cF^2_{g^j(x)}(g^j(x),r_1)$ are curves length bounded by some uniform constant $C$
if $\dim E^2=1$, and they reduce to points if $\dim E^2=0$. In the first case one can easily see
that $r_n(B^{\pm}_{\infty}(x,\vep),\beta) \le Cn/\beta$ for every $n\ge 1$ and $\beta>0$,
whereas, in the second case $r_n(B^{\pm}_{\infty}(x,\vep),\beta)=1$. So, in either case,
$r(B^{\pm}_{\infty}(x,\vep), \beta)=0$ for every $\beta>0$.
In this way, we have reduced the proof of Theorem~\ref{t.entropyexpansive} to proving
Proposition~\ref{p.dominatedsplitting}.

\subsection{Proof of the main results}\label{ss.proofs}

We are in a position to deduce all our main results.
As mentioned before, Corollary~\ref{c.principal extension} follows from
Theorem~\ref{t.entropy expansiveness} and a result in~\cite{BFF02}.
Theorem~\ref{t.upper-semi-continuity of topological entropy} is a direct consequence of
Lemma~\ref{l.upper continuous topological entropy} and Corollary~\ref{c.entropyexpansive}.
Corollary~\ref{c.main} follows immediately from Theorem~\ref{t.entropy expansiveness},
as we also observed before. Theorem~\ref{t.entropy expansiveness} is a corollary of
Proposition~\ref{total implies entropy expansive} and Corollary~\ref{c.entropyexpansive}.
Finally, to prove Theorem~\ref{t.entropy conjecture} one can argue as follows.
Given any $f\in \Diff(M)\setminus \HTclos$, let $(f_n)_n$ be a sequence of $C^\infty$
diffeomorphisms converging to $f$ in the $C^1$ topology. We may assume that every $f_n$
belongs to the isotopy class of $f$, so that $\spec(f_n)=\spec(f)$. Then, by upper semi-continuity
of the topological entropy (Theorem~\ref{t.upper-semi-continuity of topological entropy})
and the main result in Yomdin~\cite{Yom87},
$$
h(f)
\geq \limsup_{n\to\infty}h(f_n)
\geq \limsup_{n\to\infty}\log\spec((f_n)_{*})
=\log\spec(f_{*}).
$$
Therefore, $f$ satisfies the entropy conjecture, as stated. This completes the proof.

Closing this section, we prove Remark~\ref{r.converse}.
If $\HTclos$ has empty interior (in the $C^1$ topology) then we may take
$\cR=\Diff(M)\setminus\HTclos$, and there is nothing  to prove.
From now on, assume that $\inter(\HTclos)$ is non-empty.
For each $k\ge 1$, define $\cR_k$ to be the set of diffeomorphisms which either are
away from tangencies, or admit a hyperbolic set of the form
\begin{equation}\label{eq.Lambda}
\Lambda \cup f(\Lambda) \cup \cdots \cup f^{m-1}(\Lambda)
\end{equation}
for some $m\ge 1$, with $f^m(\Lambda)=\Lambda$ and $\diam(f^j(\Lambda))< 1/k$
for every $j$. Since hyperbolic sets are stable under small perturbations of the
diffeomorphism, and the diameter remains essentially unchanged, $\cR_k$ is
a $C^1$ open set. Moreover, $\cR_k$ is $C^1$ dense in $\Diff(M)$. Indeed,
consider any $g\in\Diff(M)$. If $g$ is away from tangencies then, by definition, it belongs
to $\cR_k$. So, we may suppose that $g\in\HTclos$. It follows from homoclinic
bifurcation theory (see, for instance, \cite[Chapter~6]{PT93}) that, given any $\vep>0$,
there exist diffeomorphisms $f$ arbitrarily close to $g$ such that $f$ admits a hyperbolic
set of the form \eqref{eq.Lambda} with $\max_j\diam(f^j(\Lambda))<\vep$.
This proves that $\cR_k$ is indeed dense, for every $n$.
Then $\cR=\cap_k \cR_k$. $\cR$ is a $C^1$ generic subset.
One can easily verify that each diffeomorphism $f\in \cR\cap \HTclos$ has a
sequence of periodic horseshoes with periodic diameters converging to $0$.
This implies that $f$ is not entropy expansive, as claimed.

\section{Proof of Proposition~\ref{p.dominatedsplitting}}\label{s.proofofdomination}

Let $f:M\to M$ be any diffeomorphism away from tangencies. We denote by $\tau(p,f)$ the
smallest period of a periodic point $p$. The logarithms of the norms of eigenvalues of
$Df^{\tau(p,f)}(p)$ are called \emph{exponents} of $f$ at the periodic point $p$.

\begin{Prop}[Wen~\cite{Wen04}]\label{p.uniform on large periodic orbits}
There are constants $\lambda_1$, $\gamma_1>0$, $L_1\ge 1$, and a neighborhood $\cU_1$ of $f$
such that, for any periodic point $p$ of any diffeomorphism $g\in \cU_1$,
\begin{enumerate}
\item there is at most one exponent in $[-\gamma_1,\gamma_1]$; if such an exponent
does exist, the corresponding eigenvalue is real and of multiplicity $1$;

\item there is an $L_1$-dominated splitting $T_{\Orb(p,g)}M=E^{cs}\oplus E^c\oplus E^{cu}$ over the
orbit of $p$, where $E^{cs}$,  $E^c$, $E^{cu}$ correspond to the sums of the eigenspaces of
$Dg_p^{\tau(p,g)}$ whose exponents fall in $(-\infty, -\gamma_1)$ and
$[-\gamma_1, \gamma_1]$ and $(\gamma_1, +\infty)$.

\item if $\tau(p,g)\geq L_1$, then
 $$
 \begin{aligned}
\frac{1}{[\tau(p,g)/L_1]} & \sum_{i=0}^{[\tau(p,g)/L_1]-1}\log\|Dg^{L_1}\mid E^{cs}_{g^{iL_1}(p)}\|<-\lambda_1
\quand \\
\frac{1}{[\tau(p,g)/L_1]} & \sum_{i=0}^{[\tau(p,g)/L_1]-1}\log\|Dg^{-L_1}\mid E^{cu}_{g^{-iL_1}(p)}\|<-\lambda_1.
\end{aligned}
$$
\end{enumerate}
\end{Prop}

Take $\lambda_1$, $\gamma_1$, $L_1$, and the neighborhood $\cU_1$ to be fixed once and for all.
Moreover, denote $K_1=\max\{|\log\|Dg^{m}(x)\|\,|: g\in \cU_1 \text{ and } x\in M \text{ and } |m|\le L_1\}$.
Let $g\in\cU_1$ and $\mu$ be any ergodic $g$-invariant probability measure.
We are going to use Ma\~n\'e's ergodic closing lemma:

\begin{Prop}[Ma\~n\'e~\cite{Man82}]\label{p.ergodic closing lemma}
Let $\mu$ be an ergodic measure of a diffeomorphism $g$.
Then there exist diffeomorphisms $g_n$, $n\ge 1$ and probability measures $\mu_n$, $n\ge 1$,
where each $\mu_n$ is $g_n$-invariant and supported on a periodic orbit $\Orb(p_n,g_n)$,
such that $(g_n)_n\to g$ in the $C^1$ topology and $(\mu_n)_n\to\mu$ in the weak$^*$ topology.
\end{Prop}

Of course, we may assume that $g_n\in\cU_1$ for all $n$.
Then, by Proposition~\ref{p.uniform on large periodic orbits}, the orbit of each $p_n$ admits an
$L_1$-dominated splitting $T_{\Orb(p_n,g_n)}M=E^1_n\oplus E^2_n\oplus E^3_n$ such that
$\dim(E^2_n)\leq 1$. Restricting to a subsequence if necessary, we may assume that the
dimensions of the subbundles  $E^i_n$ are independent of $n$. The fact that $(\mu_n)_n$
converges to  $\mu$ in the weak$^*$ topology implies that any Hausdorff limit of the sequence
$(\Orb(p_n,g_n))_n$ contains the support of $\mu$. It follows,
that the support admits an $L_1$-dominated splitting $T_{\supp\mu}M=E^1\oplus E^2\oplus E^3$
with $\dim(E^2)\leq 1$ (see remark at the end of page 288 in \cite{Beyond}).
This gives the first claim in Proposition~\ref{p.dominatedsplitting}.
For the proof of \eqref{eq.asymptotic} it is convenient to distinguish two cases.

\subsection{Measures with large support}\label{ss.large}

Take $\lambda_0\in(0,\lambda_1)$ and $\cU_0=\cU_1$ and $L_0$ to be an appropriately large
multiple of $L_1$ (to be chosen along the way). We are going to prove that \eqref{eq.asymptotic}
holds for every ergodic invariant probability measure $\mu$ whose support contains at least
$L_1$ points.  Let $(g_n)_n$ and $(\mu_n)_n$ be as in the ergodic closing lemma.
The assumption $\#\supp\mu\ge L_1$ implies that $\tau(p_n,g_n)\ge L_1$ for arbitrarily large $n$.
Then, restricting to a subsequence if necessary, we may assume that $\tau(p_n,g_n)\geq L_1$
for every $n$. Thus, we are in a position to use part (3) of
Proposition~\ref{p.uniform on large periodic orbits}.

\begin{Lem}\label{l.claim1}
There exists $\theta_0>0$, and for any $n\ge 1$ there exists $\Lambda_n\subset \Orb(p_n,g_n)$,
such that $\mu_n(\Lambda_n)\ge\theta_0$ and
$$
\frac{1}{k}\sum_{i=1}^{k}\log\|Dg^{-L_1}\mid E^{3,n}_{g_n^{iL_1}(q)}\|\le-\lambda_0
\quad\text{for every $q\in \Lambda_n$ and $k\ge 1$.}
$$
\end{Lem}

\begin{proof}
We are going to apply Lemma~\ref{Pliss} to $a_i=\log\|Dg^{-L_1}\mid E^{3,n}_{g^{(i-1)L_1}(p_n)}\|$
for $i=1, \dots, N$, where $N\ge 1$ is some large integer (precise conditions are stated along
the way). Take $a_*=-K_1$ and $c_2=-\lambda_1$ and $c_1=-\lambda_0$ and
$\theta=(\lambda_1-\lambda_0)/(K_1-\lambda_0)$.
The assumption of the lemma is a direct consequence of part (3) of
Proposition~\ref{p.uniform on large periodic orbits}, as long as we choose $N$
to be a multiple $[\tau(p_n,g)/L_1]$.
The conclusion of the lemma yields $1 \le n_1 < \cdots < n_l \le N$ with $l>\theta N$ such that,
for every $j=1, \dots, l$,
$$
\sum_{i=m}^{n_j-1} \log\|Dg^{-L_1}\mid E^{3,n}_{g_n^{-iL_1}(q)}\|\le -(n_j-m)\lambda_0
\quad\text{for all } 0 \le m < n_j.
$$
Denoting $q_{n,j}=g^{-n_jL_1}(p_n)$, this may be rewritten as
\begin{equation}\label{eq.rewritten}
\sum_{i=1}^{k} \log\|Dg^{-L_1} \mid E^{3,n}_{g_n^{iL_1}(q_{n,j})}\|\le -k\lambda_0
\quad\text{for all } 1 \le k \le n_j.
\end{equation}
Assume that $n_j\ge\tau(p_n,g)$. Observing that $g^{\tau(p_n,g)L_1}(q_{n,j})=q_{n,j}$,
one easily deduces that the inequality \eqref{eq.rewritten} holds for every $1 \le k < \infty$.
This means that the conclusion of the lemma holds for every point $q$ in
$$
\Lambda_n = \{g^{-n_jL_1}(p_n): \tau(p_n,g) \le n_j < N\}.
$$
Observe that $\#\{j: \tau(p_n,g) \le n_j < N\} > \theta N - \tau(p_n,g)$,
but different values of $n_j$ may yield the same point in $\Lambda_n$.
Take $N$ to be some large multiple $\kappa\tau(p_n,g)$ of the period.
Then $N$ is also a multiple of the smallest period
$\tau(p_n,g^{L_1})=\tau(p_n,g)/\gcd(L_1,\tau(p_n,g))$, of $p_n$ relative to
the iterate $g^{L_1}$. Hence,
$$
\begin{aligned}
\#\Lambda_n
& \ge \frac{\theta N - \tau(p_n,g)}{N/\tau(p_n,g^{L_1})}
= \frac{\theta\kappa-1}{\kappa\gcd(L_1,\tau(p_n,g))}\,\tau(p_n,g) \\
& \ge \frac{\theta\kappa-1}{\kappa L_1}\,\tau(p_n,g)
\ge \frac{\theta}{2L_1}\,\tau(p_n,g),
\end{aligned}
$$
as long as $\kappa$ is large enough. Then
$\mu_n(\Lambda_n) = {\#\Lambda_n}/{\tau(p_n,g)} \ge \theta/(2L_1)$.
The proof of the lemma is complete.
\end{proof}

Let us proceed with the proof of \eqref{eq.asymptotic} in the case $\supp\mu\ge L_1$.
Restricting to a subsequence if necessary, we may assume that $(\Lambda_n)_n$ converges
to some compact set $\Lambda$ in the Hausdorff topology.
Since $(\mu_n)_n$ converges to $\mu$ in the weak$^*$ topology, we have that
$\mu(\Lambda)\ge\theta_0$. Moreover,
\begin{equation}\label{eq.bound0}
\frac 1k \sum_{i=1}^{k}\log\|Dg^{-L_1}\mid E^{3}_{g^{iL_1}(y)}\|\le -\lambda_0
\quad\text{for every } k\ge 1 \text{ and } y\in\Lambda.
\end{equation}
By ergodicity, for $\mu$-almost every $x$, there exists $n(x)\ge 1$ such
that $g^{n(x)}(x)\in \Lambda$.
Take $L_0=\kappa L_1$ for some large $\kappa\ge 1$ and denote $j_0=[n(x)/L_0]$. Clearly
\begin{equation}\label{eq.bound1}
\sum_{j=1}^{j_0} \log\|Dg^{-L_0}\mid E^3_{g^{jL_0}(x)}\| \le j_0 \kappa K_1.
\end{equation}
Let $j_1=[(n(x)-j_0L_0)/L_1]$ and $l_1=n(x)-j_0L_0-j_1L_1$.
By construction, $j_1\in[0,\kappa)$ and $l_1\in[0,L_1)$. Let us write
$g^{-L_0} = g^{-l_1} \circ \big(g^{-L_1}\big)^{\kappa} \circ g^{l_1}$.
Then, for every $j > j_0$, the expression $\log\|Dg^{-L_0}\mid E^3_{g^{jL_0}(x)}\|$
is bounded by
\begin{equation}\label{eq.bound2}
\begin{aligned}
 \sum_{i=1}^{\kappa} \log\| & Dg^{-L_1} \mid E^3_{g^{(j-1)L_0+l_1+iL_1}(x)}\| + 2 K_1\\
& = \sum_{i=(j-1-j_0)\kappa+(1-j_1)}^{i=(j-1-j_0)\kappa+(\kappa-j_1)}\log\|Dg^{-L_1}\mid E^3_{g^{i L_1}(y)}\|
+ 2 K_1,
\end{aligned}
\end{equation}
where $y=g^{n(x)}(x)$. Adding \eqref{eq.bound1} to the sum of \eqref{eq.bound2} over
$j=j_0+1, \dots, n$, we find that $\sum_{j=1}^n\log\|Dg^{-L_2} \mid E^3_{g^{jL_2}(x)}\|$
is bounded by
$$
\begin{aligned}
j_0 \kappa K_1 + & \sum_{i=(1-j_1)}^{(n-j_0)\kappa+(\kappa-j_1)}\log\|Dg^{-L_1}\mid E^3_{g^{iL_1}(y)}\|
+ 2 K_1 n \\
& \leq  \big(j_0 \kappa + j_1) K_1
+ \sum_{i=1}^{(n-j_0)\kappa-j_1}\log\|Dg^{-L_1}\mid E^3_{g^{iL_1}(y)}\|
+ 2 K_1n.
\end{aligned}
$$
Consequently,
$$
\begin{aligned}
\limsup_{n\to\infty} \frac{1}{n} \sum_{j=1}^{n} \log\| & Dg^{-L_2}\mid E^3_{g^{jL_2}(x)}\| \\
&  \leq \kappa \limsup_{k\to\infty} \frac{1}{k}\sum_{i=1}^{k}\log\|Dg^{-L_1}\mid E^3_{g^{iL_1}(y)}\|
+ 2 K_1.
\end{aligned}
$$
According to \eqref{eq.bound0}, the right hand side is bounded by $-\kappa\lambda_0 + 2 K_1 \le -\lambda_0$,
as long as we choose $\kappa$ sufficiently large. This completes the proof of
\eqref{eq.asymptotic} in this case.

\subsection{Measures with small support}\label{ss.small}

Finally, we extend the claims in \eqref{eq.asymptotic}  to ergodic measures supported on
periodic orbits with period smaller than $L_1$. We need slightly more precise choices of
$\lambda_0$, $L_0$, and $\cU_0$, than in the previous section. These are made precise
along the way. Let $\Per(f,L_1)$ be the (compact) set of periodic points $p$ of $f$
such that $\tau(p,f)<L_1$.

\begin{Lem}\label{l.claim4}
There is a positive integer $m>0$, such that for any $p\in \Per(f,L_1)$
there exist $m_{\pm}(p) \in \{1, \dots, m\}$ satisfying
$$
\log\|Df^{m_+(p)\tau(p,f)} \mid E^1_p\|<0
\quand
\log\|Df^{-m_-(p)\tau(p,f)} \mid E^3_p\|<0.
$$
\end{Lem}

\begin{proof}
We explain how to find $m_+$ satisfying the first claim; the argument for the second claim is analogous.
Suppose that for every $m\ge 1$ there is $p_m\in\Per(f,L_1)$ such that
$\log\|Df^{n\tau(p_m,f)} \mid E^1_{p_m}\|\geq 0$ for all $1\leq n \leq m$.
Restricting to a subsequence if necessary, we may suppose that the
$L_1$-dominated splittings $T_{\Orb(p_m,f)}M = E_m^1 \oplus E_m^2 \oplus E_m^3$
are such that the dimensions of the subbundles $E^j_m$ are independent of $m$.
Analogously, we may suppose that the periods $\tau(p_m,f)$ are independent of $m$
and $(p_m)_n$ converges to some $p\in M$.
Then $p$ is periodic, with $\tau(p,f)=\tau(p_m,f)$, and there is an $L_1$-dominated
splitting $T_{\Orb(p,f)} M = E^1 \oplus E^2 \oplus E^3$ with $\dim E^j=\dim E_m^j$.
On the one hand, by continuity,
\begin{equation}\label{eq.frombelow}
\log\|Df^{nL_1} \mid E^1_{p}\|\geq 0
\quad\text{for any $n \ge 1$.}
\end{equation}
On the other hand, all the exponents of $Df^{\tau(p_m,f)} \mid E^1_{p_m}$
are bounded above by $-\gamma_1$ and so the same is true for the exponents
of $Df^{\tau(p,f)} \mid E^1_{p}$. It follows that
$$
\lim_{n\to\infty} \log\|Df^{n\tau(p,f)} \mid E^1_{p}\| = -\infty,
$$
which contradicts \eqref{eq.frombelow}. This contradiction proves the claim.
\end{proof}

Lemma~\ref{l.claim4} implies that if $L_0\ge 1$ is chosen to be a multiple
of $m!L_1!$ then
$$
\log\|Df^{L_0}\mid E^1_x\|<0 \quand \log\|Df^{-L_0} \mid E^3_x\|<0
$$
for every $x\in \Per(f,L_1)$. Define
$$
\lambda_*=-\max\{\log\|Df^{L_0}\mid E^1_x\|, \log\|Df^{-L_0} \mid E^3_x\|: p\in\Per(f,L_1)\}.
$$
Notice that $\lambda_*>0$, since $\Per(f,L_1)$ is compact. Moreover, by definition
\begin{equation}\label{eq.Lzero}
\log\|Df^{L_0} \mid E^1_x\|\le-\lambda_*
\quand
\log\|Df^{-L_0} \mid E^3_x \|\le-\lambda_*
\end{equation}
for all $x\in\Per(f,L_1)$. Clearly, the map
$g \mapsto \Per(g,L_1)$ is upper semi-continuous: for any neighborhood $U_0$ of
$\Per(f,L_1)$, we have $\Per(g,L_1) \subset U_0$ for every $g$ in a neighborhood of $f$.
Reducing $\cU_0$ if necessary, we may assume that this holds for every $g\in\cU_0$.
Choose $\lambda_0\in(0,\lambda_*)$. Taking some small $\delta>0$ and shrinking
$\cU_0$ and $U_0$ if necessary,
\begin{itemize}
\item[(a)] for any $g\in\cU_0$ and $x$, $y\in M$ with $d(x,y)<\delta$, we have
$$
\begin{aligned}
|\log\|Df^{L_0} \mid E^1_x\| & -\log\|Dg^{L_0} \mid E^1_y\|| < \lambda_*-\lambda_0 \quand \\
|\log\|Df^{-L_0} \mid E^3_x\| & -\log\|Dg^{-L_0} \mid E^3_y\|| <\lambda_*-\lambda_0.
\end{aligned}
$$
\item[(b)] for any $g\in \cU_0$ and $y\in U_0$, there exists $x\in \Per(f,L_1)$ such that
$$
d(f^{jL_0}(x),g^{jL_0}(y))<\delta \quad \text{for all $|j| \le L_1!$.}
$$
\end{itemize}
Fix $g\in \cU$ and $q\in \Per(g,L_1)\subset U_0$. By (b), there exists $p\in \Per(f,L_1)$
such that
$$
d(f^{jL_0}(p),g^{jL_0}(q))<\vep \quad\text{whenever $|j| \le L_1!$.}
$$
The periods $\tau(p,f)$ and $\tau(q,g)$ need not be the same. Combining
(a)-(b) with \eqref{eq.Lzero}, we get that
\begin{equation}\label{eq.forq}
\frac{1}{n}\sum_{i=1}^{n}\log\|Df^{L_0} \mid E^1_{f^{-iL_0}(q)}\|\le-\lambda_0.
\end{equation}
for any $1\leq n \leq L_1!$. Since $\tau(q,g) < L_1!$, it follows that \eqref{eq.forq}
holds for every $n\ge 1$. The proof of the claim about $\log\|Df^{L_0} \mid E^3\|$
is analogous. This finishes the proof of Proposition~\ref{p.dominatedsplitting}.


\end{document}